\documentclass[reqno,times,twoside, 11pt]{amsart}

\usepackage{hyperref}
\usepackage{enumerate,fullpage}
\usepackage{amssymb,amsmath,graphicx,amsthm}
\usepackage{color}
\usepackage{anysize}
\marginsize{1in}{1in}{1in}{1in}

\allowdisplaybreaks

\newcommand{\ba}{\begin{eqnarray}}
\newcommand{\ea}{\end{eqnarray}}

\usepackage{amsfonts,amsmath,amsthm,mathrsfs}
\usepackage{amssymb}
\usepackage{color}
\usepackage{enumerate}
\usepackage{graphics,verbatim,url}
\usepackage{index}

\newcommand{\A}{\mathcal{A}}

\newcommand{\be}{\begin{eqnarray*}}
	\newcommand{\en}{\end{eqnarray*}}
\newcommand{\bes}{\begin{eqnarray}}
	\newcommand{\ens}{\end{eqnarray}}

\newcommand{\lf}{\left}
\newcommand{\rt}{\right}

\usepackage{epic, curves}
\def\nn{\nonumber}
\def  \kg {L^\infty\lf(0,T; 		 \mathcal{Z}_{\gamma, TA_1 }(\Omega)\rt)}
\newcommand{\al}{\alpha}

\newcommand{\ep}{\epsilon}
\newcommand{\bn}{ \beta_{\bf n} }

\newtheorem{theorem}{Theorem}[section]

\newtheorem{corollary}{Corollary}[section]
\newtheorem{definition}{Definition}[section]
\newtheorem{lemma}{Lemma}[section]

\newtheorem{remark}{Remark}[section]

\def\bq{\begin{equation}}
	\def\eq{\end{equation}}
\def\bqq{\begin{eqnarray*}}
	\def\eqq{\end{eqnarray*}}

\def\nn{\nonumber}

\begin{document}
	\title{ On a backward problem for multidimensional  Ginzburg-Landau equation with  random data  }         
	\author{Mokhtar Kirane, Erkan Nane  and Nguyen Huy Tuan}          
	\address{M. Kirane}
	\address{LaSIE, Facult\'{e} des Sciences et Technologies, Universi\'{e} de La Rochelle,  Avenue M. Cr\'{e}peau, 17042 La Rochelle, France}
	\email{mokhtar.kirane@univ-lr.fr }
	\address{E. Nane}
	\address{Department of Mathematics and Statistics, Auburn University, Auburn, USA }
	\email{ezn0001@auburn.edu }
	\address{N.H. Tuan}
	\address{Department of Mathematics and Informatics,  University of Science, Viet Nam National University VNU-HCMC, Viet Nam }
	\email{nguyenhuytuan@tdt.edu.vn or thnguyen2683@gmail.com  }
\maketitle

		\begin{abstract}
		In this paper, we consider a backward in time problem for Ginzburg-Landau equation in multidimensional domain associated with some random data. The problem is ill-posed in the sense of Hadamard. To regularize the instable solution, we develop a new regularized method combined with  statistical approach to solve this  problem. We prove a upper  bound,  on the rate of convergence of the mean integrated squared error in $L^2 $ norm and $H^1$ norm.
		\end{abstract}

		
		\section{Introduction}

		In this paper we consider the backward  problem of finding ${\bf u}({\bf x},0 )$ for Ginzburg-Landau equation
		\begin{equation}
		\label{problem555}
		\left\{\begin{array}{l l l}
		{\bf u}_t - \Lambda(t) \Delta  	{\bf u} & = 	{\bf u}-{\bf u}^3+G(	{\bf x},t), & \qquad (	{\bf x},t) \in \Omega\times (0,T),\\
		{\bf u}(	{\bf x},t)&=0, & \qquad 	{\bf x} \in \partial {\Omega},\\
		{\bf u}(	{\bf x},T) & = H(	{\bf x}), & \qquad 	{\bf x} \in {\Omega},
		\end{array}
		\right.
		\end{equation}
		where $\Lambda \in C([0,T])$
		and $H$ is  in $L^2(\Omega)$.
		Here  the domain $\Omega=(0,\pi)^d$ is a   subset of  $\mathbb{R}^d$ and { ${\bf x}:=(x_1,...x_d)$.}    The function  $G$ is called the source function that satisfies the usual Lipschitz  continuity and growth conditions. The function $H$ is given and  is often called   a final value data.  The Ginzburg-Landau  equation has been applied
	in various areas in physics, including phase
	transitions in non-equilibrium systems, instabilities in hydrodynamic systems, chemical turbulence, and thermodynamics (see \cite{Ames,Ka}).

	 It is well-known that  the backward in time   problem for Ginzburg-Landau equation  is severely ill--posed in the sense of Hadamard \cite{Ames} and \cite{hada}.  Hence  solutions do not always exist, and in the case of existence,
		the solutions  do not  depend continuously on the   data.
		In \cite{Ames}, the author considered continuous dependence of the solutions    on the parameter of
		a Ginzburg--Landau equation.
		 In fact, from small noise contaminated physical measurements of $H$ and $G$,
		the corresponding solutions might have large errors. In practice, if we  measure the function
		$H({\bf x})$ and $ G({\bf x},t)$ at fixed points ${\bf x}_{\bf i }$
		with index  ${\bf i}=(i_1,i_2,...i_d) \in \mathbb{N}^d, 1 \le i_k \le n_k $ for $k=\overline{1,d}$ where
		\begin{equation} \label{model1}
		{\mathbf x}_{\bf i }=( x_{i_1},...x_{i_d})=\Big(\frac{\pi(2i_1-1)}{2n_1}, \frac{\pi(2i_2-1)}{2n_2},...\frac{\pi(2i_d-1)}{2n_d}  \Big),\quad i_k= \overline {1,n_k},\quad k=\overline {1,d}.
		\end{equation}
		then we obtain a set of values
		\begin{equation*}
			\widetilde D_{\bf i}=\widetilde D_{i_1,i_2,...i_d} \approx  H	( x_{i_1},...x_{i_d}),~~	\widetilde G_{\bf i}(t)=\widetilde G_{i_1,i_2,...i_d}(t) \approx G	( x_{i_1},...x_{i_d},t).
		\end{equation*}
	The points  $ 	{\mathbf x}_{\bf i } , {\bf i}=(i_1,i_2,...i_d),~\quad i_k= \overline {1,n_k},\quad k=\overline {1,d} $ are called (non-random) design points. The real measurements are always observed with errors.
		If the errors are generated from uncontrollable sources  such as wind, rain, humidity, etc, then the model is random. 	We consider the following   nonparametric regression model of  data as follows
		\begin{align} \label{model2}
		\widetilde D_{\bf i}=\widetilde D_{i_1,i_2,...i_d}:&= H	( x_{i_1},...x_{i_d})+ \Lambda_{i_1,i_2,...i_d} \Upsilon_{i_1,i_2,...i_d}=H({\bf x}_{\bf i })+\Lambda_{\bf i}\Upsilon_{\bf i}
		\end{align}
		\begin{align} \label{model3}
		\widetilde G_{\bf i}(t)=\widetilde G_{i_1,i_2,...i_d}(t):&=G	( x_{i_1},...x_{i_d},t)+\vartheta  \Psi_{i_1,i_2,...i_d}(t)=G({\bf x}_{\bf i },t)+\vartheta  \Psi_{\bf i}(t),
		\end{align}
		for $ i_k= \overline {1,n_k},\quad k=\overline {1,d}$.   Here $\Upsilon_{\bf i}:=\Upsilon_{i_1,i_2,...i_d} \sim \mathcal{N} (0,1)$  and $\Psi_{\bf i}(t):=\Psi_{i_1,i_2,...i_d}(t)$ are Brownian motions. {  Here $\Lambda_{i_1,i_2,...i_d}$ and $\vartheta$ are  positive constants which are bounded by a positive constant $V_{max}$.} We assume furthermore that they are  mutually independent.
		
	A well known fact is that, when  the "noise"   in these models  are
		modeled as a random quantity, the convergence of estimators  $\widetilde {\bf u}({\bf x},0)$ of  ${\bf u}({\bf x},0)$ should be studied
		by  statistical methods.  Methods applied to  the deterministic cases cannot  be applied directly for this case.   The main idea in using  the random noise is of finding suitable  estimators
		$\widetilde {\bf u}({\bf x},0)$ and to consider the expected square error $\mathbb{E} \Big[\|\widetilde {\bf u}({\bf x},0) - {\bf u}({\bf x},0)\|^2_{L^2}\Big]$ in a suitable space, also  called the mean integrated square error (MISE).

 The inverse problem with random noise has a long history.
		 The
	  backward  problem  for linear parabolic equation is a special form of statistical inverse problems and it can be transform  by a linear operator with random noise
		\begin{equation}
			u_T=K u_0 + \text{"noise"}.  \label{K}
		\end{equation}
		where $K$ is a bounded linear operator that  does not have a continuous inverse.
		The latter model  implies that
	some  well-known methods  including spectral cut-off (or called truncation method)  \cite{Bi2,Cavalier,Mair,Hohage,Trong},  the Tiknonov method \cite{Cox}, iterative regularization methods \cite{Engl} can be used.

		In this paper, since the model in \eqref{problem555} is nonlinear, we can not transform it into the operator  defined  in equation \eqref{K}.
		This makes the considering the nonlinear problem \eqref{problem555} is more challenging. Another  difficulty arises when
	  $\Lambda$ is noisy by random observation
		\begin{equation}
		\Lambda_\ep(t)= \Lambda(t)+ \ep \overline  {\xi}(t), \label{Lamb}
		\end{equation}
		where $\ep$ is deterministic noise level and $\overline  {\xi}(t)$  is Brownian motion.  If $\Lambda(t)$ is a constant (independent of $t$) then we can apply    well-known  methods such as the  spectral method (see section 3 \cite{Tuan3})  for solving the problem \eqref{problem555}. However, when  $\Lambda$ depends on $t$ and is noisy  as in equation \eqref{Lamb}, the problem is more challenging. It is not possible to  approximate the solution of problem \eqref{problem555} using the spectral method.   	Until now, to the best of the authors' knowledge,  there
		does not exist any results  for approximating the solution of the Problem  \eqref{problem555} with the random model  \eqref{model2},\eqref{model3}, \eqref{Lamb}.  This is  our motivation in the  present paper.
		
		Our main goal in this paper is to provide  a new regularized method to give  a regularized solution that is  called estimators  for approximating ${\bf u}({\bf x},t),~0\le t <T$.
The backward problem  for Ginzburg--Landau equation  with random data has  not been  studied before.
Our main idea in this paper is that of  applying  a modified  Quasi-reversibility method as in  Lions \cite{Lion}. First, we approximate $H$ and $G$ by the  approximating functions $\widehat H_{\beta_{\bf n}}$ and $\widehat G_{\beta_{\bf n}}$  that are  defined in Theorem \eqref{theorem2.1}.
Next, our task is  of
finding the approximating operator for  $  \Lambda(t) \Delta$.  We will not  approximate directly the time dependent operator $\A(t)=\Lambda(t) \Delta$ as introduced in \cite{Lion}.
We introduce  a new approach by
giving  the unbounded time independent operator ${\bf P}$  that is defined in Lemma \eqref{lem3.1}.  Then, we  approximate   ${\bf P}$
by a bounded operator ${\bf P_{\rho_n}}$,  in order to establish an approximation for  the regularized problem \ref{regu333}.  Here  $\beta_n $  satisfies that   $\lim_{n\to \infty} \beta_n
=+\infty$, and we choose
$\rho_{ n}$  that depends on $\beta_n$ suitably  to obtain the convergence rate.
 In contrast to the initial  value problem, for  the final value (inverse) problem, we  need to assume that  the  problem \eqref{problem555} has a unique solution ${\bf u}$.   In particular,   the main purpose in  our error estimates is to show that  the norm of difference between the regularized solution of the problem \eqref{regu333} and the  solution  of the problem \eqref{problem555}   in $L^2(\Omega)$ and $H^1(\Omega)$  tends to zero when $ |{\bf n}|=\sqrt{n_1^2+...+n_d^2} \to +\infty$.

		\section{Constructing a function from discrete random data}
		\setcounter{equation}{0}
		In this section, we develop  a new theory for constructing a function in $L^2(\Omega)$ from the  given discrete random data.
		
		{	We first  introduce notation, and then we state the main results of this paper.

We will occasionally use the following  Gronwall's inequaly in this paper.
			\begin{lemma}
			Let $ b: [0,T] \to \mathbb{R}^+ $ be a continuous function  and $ C,D >0$ be constants that are  independent of $t$,  such that
			\begin{equation*}
			b(t) \le C+ D \int_{t}^T b(\tau)d\tau,\ \ \ t>0.
			\end{equation*}
			Then we have
			\begin{equation*}
			b(t) \le C e^{D(T-t)}.
			\end{equation*}
			\end{lemma}
			Next we define fractional powers of the Dirichlet Laplacian
			\begin{equation*}
				Af:= -\Delta f.
			\end{equation*}
			Since $A$ is a linear, densely defined self-adjoint and positive definite
			elliptic operator on the connected bounded domain  $\Omega $ with
			Dirichlet  boundary condition, {  using spectral theory, it is easy to show that  the eigenvalues of $A$ are given by}  $\lambda_{\bf p}=|{\bf p}|^2= p_1^2+p_2^2+\cdots +p_d^2  $.
			The corresponding eigenfunctions are denoted respectively by
			\begin{equation}\label{eigenfunctions-laplacian}
				\psi_{\bf p}({\bf x})=\bigg(\sqrt{\frac{2}{\pi}}\bigg)^d\sin (p_1 x_1)\sin(p_2x_2)\cdots \sin(p_d x_d).
			\end{equation}
			Thus the eigenpairs $(\lambda_{\bf p},\psi_{\bf p})$,
			$p\in \mathbb{N}^d$, satisfy
			\[
			\begin{cases}
			A \psi_{\bf p}({\bf x})
			=
			-\lambda_{\bf p} \psi_{\bf p}({\bf x}),
			\quad & {\bf x} \in \Omega \\
			\psi_{\bf p}({\bf x})
			=
			0,
			\quad & {\bf x}\in \partial \Omega.
			\end{cases}
			\]
			The functions $\psi_{\bf p}$ are normalized so that
			$\{\psi_{\bf p}\}_{{\bf p}\in \mathbb{N}^d}$ is an orthonormal basis of $L^2(\Omega)$.\\
		}
		We will use the following notation:
		$|{\bf p}| =|(p_1,\cdots, p_d)|= \sqrt{p_1^2+...+p_d^2}$, $|{\bf n}|=|(n_1,\cdots, n_d)|= \sqrt{n_1^2+...+n_d^2}$.

		\begin{definition}
			For $\gamma>0$, we define
			\begin{equation}
				\mathcal{H}^\gamma(\Omega):= \Big\{ h \in L^2(\Omega):  \sum_{p_1=1}^\infty...\sum_{p_d=1}^\infty |{\bf p}|^{2\gamma}  <h, \psi_{\bf p}>^2 ~ <\infty \Big\}.
			\end{equation}
			The norm on $\mathcal{H}^\gamma(\Omega)$ is defined by
			\begin{equation}
				\|h\|^2_{	\mathcal{H}^\gamma(\Omega)}:=	\sum_{p_1=1}^\infty...\sum_{p_d=1}^\infty |{\bf p}|^{2\gamma}  <h, \psi_{\bf p}>^2.	
			\end{equation}
			
		\end{definition}

		{   For any   Banach space $X$, we denote by $L_{p}\left(0,T;X\right)$,  the Banach space of  measurable real functions
			$v:(0,T)\to X$  such that
			\begin{align*}
				\left\Vert v\right\Vert _{L^{p}\left(0,T;X\right)}=\left(\int_{0}^{T}\left\Vert v\left(\cdot,t\right)\right\Vert _{X}^{p}dt\right)^{1/p}<\infty,\quad 1\le p<\infty,
			\end{align*}
			\begin{align*}
				\left\Vert v\right\Vert _{L^{\infty}\left(0,T;X\right)}= \text{esssup}_{0<t<T}\left\Vert v\left(\cdot,t\right)\right\Vert _{X}<\infty,\quad p=\infty.
			\end{align*}
			
		}		
	
		Let $\beta:\mathbb{N}^d\to \mathbb{R}$ be a function.
		We state  the next result  which gives  error estimate between $H$ and $\widehat H_{\beta_{\mathbf n}}$, and  error estimate between $\widehat G_{\beta_{\bf n}} $  and $G$.

			\begin{theorem}[Theorem 2.1 in Kirane et al. \cite{Tuan3}]  \label{theorem2.1}
				Define the set $	\mathcal{W}_{   \beta_{\bf n}}$ for any ${\bf n}=(n_1,..n_d)\in \mathbb{N}^d$
				\begin{equation}
				\mathcal{W}_{   \beta_{\bf n}}=\mathcal{W}_{  \beta_{\bf n}} = \Big\{ {\bf p} =(p_1,...p_d) \in \mathbb{N}^d  : |{\bf p}|^2= \sum_{k=1}^d p_k^2 \le \beta_{\bf n}= \beta(n_1,...n_d)  \Big\}
				\end{equation}
				where $\beta_{\bf n}$  satisfies
				\begin{equation*}
				\lim_{|{\bf n}| \to +\infty} \beta_{\bf n}=+\infty.
				\end{equation*}
				
				For a given ${\bf n}$ and $\beta_{\bf n}$ we define  functions that are approximating  $H, G$ as follows
				\begin{equation}
				\widehat H_{\beta_{\bf n}}({\bf x}) = \sum_{{\bf p} \in  	\mathcal{W}_{  \beta_{\bf n}}   } \Bigg[\frac{\pi^d}{\prod_{k=1}^d n_k } \sum_{i_1=1}^{n_1}...\sum_{i_d=1}^{n_d} \widetilde D_{i_1,i_2,...i_d} \psi_{\bf p}	( x_{i_1},...x_{i_d}) \Bigg] \psi_{\bf p}({\bf x})
				\end{equation}
				and
				\begin{equation}
				\widehat G_{\beta_{\bf n}}({\bf x},t) = \sum_{p \in 	\mathcal{W}_{  \beta_{\bf n}}  } \Bigg[\frac{\pi^d}{\prod_{k=1}^d n_k } \sum_{i_1=1}^{n_1}...\sum_{i_d=1}^{n_d} \widetilde G_{i_1,i_2,...i_d}(t) \psi_{\bf p}	( x_{i_1},...x_{i_d}) \Bigg] \psi_{\bf p}({\bf x}).
				\end{equation}
				Let $\mu=(\mu_1,...\mu_d) \in {\mathbb R}^d$ with $\mu_k >\frac{1}{2} $ for  any $ k=\overline {1,d}$. Let us choose $\mu_0 \ge { d \max(\mu_1,...\mu_d )}$.
				If  $H \in \mathcal{H}^{\mu_0}(\Omega) $ and $G \in L^\infty (0,T;\mathcal{H}^{\mu_0}(\Omega) )$ then the following estimates hold
				\begin{eqnarray}
				\begin{aligned}
				&	{\bf E}\Big\| \widehat H_{\beta_{\bf n}} -H  \Big\|^2_{L^2(\Omega)} \le  \overline C (\mu_1,...\mu_d, H) \beta_{\bf n}^{d/2} \prod_{k=1}^d  ( n_k)^{-4\mu_k}+{ 4\beta_{\bf n}^{-\mu_0}} \Big\| H \Big\|^2_{\mathcal{H}^{\mu_0}(\Omega) },\nn\\
				&	{\bf E}\Big\| \widehat G_{\beta_{\bf n}}(.,t) -G(.,t)  \Big\|^2_{L^\infty(0,T; L^2(\Omega))} \le  \overline C (\mu_1,...\mu_d, H) \beta_{\bf n}^{d/2} \prod_{k=1}^d  ( n_k)^{-4\mu_k}+ {4\beta_{\bf n}^{-\mu_0} }\Big\| G \Big\|^2_{L^\infty (0,T;\mathcal{H}^{\mu_0}(\Omega) ) }	,
				\end{aligned}
				\end{eqnarray}	
				where
				\begin{equation*}
				\overline C (\mu_1,...\mu_d, H)= 8\pi^d  V_{max}^2  \frac{2 \pi^{d/2}}{ d  \Gamma(d/2)} + \frac{ 16 \mathcal{C}^2 (\mu_1,...\mu_d) \pi^{d/2}}{ d  \Gamma(d/2)}  \Big\| H \Big\|^2_{\mathcal{H}^{\mu_0}(\Omega) }.
				\end{equation*}	
			\end{theorem}

		\begin{corollary}[ Corollary 2.1 in Kirane et al. \cite{Tuan3}]  \label{corollary2.1}
			Let  $H, G$ be as in  Theorem \eqref{theorem2.1}. Then  the term
			 $ {\bf E}	\Big\| \widehat H_{\beta_{\bf n}}-H \Big\|_{L^2(\Omega)}^2+T {\bf E}	\Big\| \widehat G_{\beta_{\bf n}}-G \Big\|_{L^\infty(0,T;L^2(\Omega))}^2$
			 is of order
			 \begin{equation*}  \label{veryimportant}
			 \max \Bigg( \frac{   \beta_{\bf n}^{d/2} }{\prod_{k=1}^d  ( n_k)^{4\mu_k} }, ~\beta_{\bf n}^{-\mu_0}  \Bigg).
			 \end{equation*}
		\end{corollary}

					\section{Backward problem for  parabolic equations with random coefficients }
					\setcounter{equation}{0}
				
				In this section, we assume that
					 $ \Lambda_\ep, \Lambda$ are  continuous functions on $[0,T]$, hence  there exist  two positive numbers $A_0, A_1$  such that
					\begin{equation} \label{A1}
					A_0 \le \|\Lambda_\ep \|_{C([0,T])}= \sup_{0 \le t \le T} |\Lambda_\ep (t)| < A_1,~~	A_0 \le \|\Lambda \|_{C([0,T])}= \sup_{0 \le t \le T} |\Lambda (t)| < A_1.
					\end{equation}
					It is well-known that the function $F(u)= u-u^3$ is a locally Lipschitz function.
				We approximate the function $F(u)=u-u^3$ by $\overline F_{Q}$ defined by
					\[
					\overline F_{Q}\left(	{ u}(	{\bf x},t) \right)
					=
					\begin{cases}
					Q-Q^3, &\quad u(	{\bf x},t) >Q,\\
					u-u^3, &\quad - Q \le u(	{\bf x},t)\le Q, \\
					- Q+Q^3,&\quad u(	{\bf x},t) < - Q,
					\end{cases}
					\]
					for all $Q>0$. 	In the sequel we use  a parameter sequence  $Q_{\bf n}:=Q(n_1,n_2,...n_d) \to +\infty $ as $|{\bf n}| \to +\infty$. So, when $ |{\bf n}|$ is large enough, we have that  $Q_{\bf n} \ge \|{\bf u}\|_{L^\infty (0,T; L^2(\Omega))} $.  Moreover, we also have
					\begin{equation} \label{local2}
					\overline  	F_{Q_{\bf n} }( {\bf u} )= F( {\bf u}  )={\bf u}- {\bf u}^3 , ~~\text{for}~~ |{\bf n}| ~~\text{large enough}.
					\end{equation}
					Using \cite{Tuan5}, we also obtain  that 	$ \overline F_{Q_{\bf n}}$ is a globally Lipschitz source  function in the following sense
					\begin{equation}  \label{local3}
					\| 	\overline   F_{Q_{\bf n}} (v_1)- 	\overline  F_{Q_{\bf n}}(v_2) \|_{L^2(\Omega)} \le  \left(2+6Q_{\bf n}^2 \right)  \|v_1-v_2\|_{L^2(\Omega)} ,
					\end{equation}	
					for any $v_1, v_2 \in L^2(\Omega)$.
					
					\begin{lemma}  \label{lem3.1}
						Define the following space of functions
						\begin{equation} \label{definitionspace}
						\mathcal{Z}_{\gamma, B }(\Omega):=\Bigg\{ f \in L^2(\Omega),  \sum_{{\bf p} \in \mathbb{N}^d  }   |{\bf p}|^{2+2\gamma} e^{2B |{\bf p}|^{2}}  \big\langle f,\psi_{\bf p}\big\rangle_{L^2(\Omega)}^2 <+\infty \Bigg\},
						\end{equation}
						for any $\gamma \ge 0$ and $B\geq 0$. Define   also the operator  ${\bf P}= A_1 \Delta $ ( $A_1$ is the upper bound in \eqref{A1})  and
						$	{\bf P}_{\rho_{\bf n}}$   is defined as follows
						\begin{align}
						{\bf P}_{\rho_{\bf n}}(v)  &=A_1 \sum_{ {\bf |p|} \le   \sqrt{\frac{\rho_{\bf n}} {  A_1 }  }  }^\infty       |{\bf p}|^{2{}} \big\langle v(x),\psi_{\bf p}\big\rangle_{L^2(\Omega)} \psi_{\bf p},
						\end{align}
						for any function $v \in L^2(\Omega)$.  Then for  any $v \in L^2(\Omega)$
						\begin{align}
						\|{\bf P}_{\rho_{\bf n}}(v)\|_{L^2(\Omega)} \le \rho_{\bf n} \|v\|_{L^2(\Omega)},  \label{estimateP1}
						\end{align}
						and for $v \in 		\mathcal{Z}_{\gamma, TA_1 }(\Omega) $
						\begin{align}
						\| {\bf P}v-{\bf P}_{\rho_{\bf n}}v\|_{L^2(\Omega)} \le  A_1 \rho_{\bf n}^{-\gamma}  e^{-T \rho_{\bf n}  }   \|v\|_{ \mathcal{Z}_{\gamma, TA_1 }(\Omega)  }. \label{lemma4.1}
						\end{align}
						
					\end{lemma}
					\begin{proof}
						First, for any $v \in L^2(\Omega)$, we have
						\begin{align}
						\|{\bf P}_{\rho_{\bf n}}(v)\|_{L^2(\Omega)}^2 &= A_1^2 \sum_{ {\bf |p|} \le   \sqrt{\frac{\rho_{\bf n}} {  A_1 }  }  }^\infty    |{\bf p}|^{4} \big\langle v(x),\psi_{\bf p}\big\rangle^2_{L^2(\Omega)} \le \rho_{\bf n}^2   \sum_{ {\bf |p|} \le   \sqrt{\frac{\rho_{\bf n}} {  A_1 }  }  }^\infty \big\langle v(x),\psi_{\bf p}\big\rangle^2_{L^2(\Omega)}= \rho_{\bf n}^2  \|v\|_{L^2(\Omega)}^2,
						\end{align}
						and
						\begin{align}
						\| {\bf P}v -{\bf P}_{\rho_{\bf n}}(v)\|_{L^2(\Omega)}^2 &= A_1^2 \sum_{ {\bf |p|} >   \sqrt{\frac{\rho_{\bf n}} {  A_1 }  }  }^\infty    |{\bf p}|^{-4\gamma}  e^{-2TA_1{|\bf p|}^2 } |{\bf p}|^{4+4\gamma} e^{2TA_1{|\bf p|}^2 }  \big\langle v(x),\psi_{\bf p}\big\rangle^2_{L^2(\Omega)}\nn\\
						& \le  A_1^2 \rho_{\bf n}^{-2\gamma}  e^{-2TA_1 \rho_{\bf n}  } \sum_{ {\bf |p|} >   \sqrt{\frac{\rho_{\bf n}} {  A_1 }  }  }^\infty  |{\bf p}|^{4+4\gamma} e^{2TA_1{|\bf p|}^2 }    \big\langle v(x),\psi_{\bf p}\big\rangle^2_{L^2(\Omega)}\nn\\
						&=  A_1^2 \rho_{\bf n}^{-2\gamma}  e^{-2T \rho_{\bf n}  }   \|v\|_{ \mathcal{Z}_{\gamma, TA_1 }(\Omega)  }^2.
						\end{align}
						
					\end{proof}	
					
					Applying a modified quasi-reversibility method as in Section 4.1 in Kirane et al. \cite{Tuan3}, we introduce the regularized solution  defined by
					\bq
					\left\{ \begin{gathered}
					\frac{\partial \widehat U^\ep_{\rho_{\bf n}, \bn}}{\partial t}-\Lambda_\ep(t) \Delta \widehat U^\ep_{\rho_{\bf n}, \bn} -{\bf P}\widehat U^\ep_{\rho_{\bf n}, \bn} +{\bf P}_{\rho_{\bf n}}\widehat U^\ep_{\rho_{\bf n}, \bn}\\
					\quad \quad  = F_{Q_{\bf n}}({\widehat U^\ep_{\rho_{\bf n},\bn} }({\bf x},t))+\widehat G_{\beta_{\bf n} }({\bf x},t) ,~~0<t<T, \hfill \\
					{\widehat U^\ep_{\rho_{\bf n},\bn} }({\bf x},t)= 0,~~{\bf x} \in \partial \Omega,\hfill\\
					{\widehat U^\ep_{\rho_{\bf n},\bn} }({\bf x},T)=\widehat H_{\beta_{\bf n}}({\bf x}). \hfill\\
					\end{gathered}  \right. \label{regu333}
					\eq
				Since the first equation of the system \eqref{regu333} contains the term  $\Lambda_\ep(t)$ which depends on $\ep$, it is suitable to  denote the solution of Problem \eqref{regu333} by $ \widehat U^\ep_{\rho_{\bf n}, \bn}$ with three variables $\rho_{\bf n}, \bn, \ep$.
					 Now, we give convergence rates between the regularized solution $\widehat U^\ep_{\rho_{\bf n},\bn}$ of Problem \eqref{regu333} and the solution ${\bf u}$ of Problem \eqref{problem555}. Furthermore, we show that $\widehat U^\ep_{\rho_{\bf n},\bn}$ converges to  ${\bf u}$ when $|{\bf n}| \to +\infty$ and $\ep \to 0$.

					\begin{theorem}  \label{theorem5.1}
						The problem \eqref{regu333}  has a unique solution  $\widehat U^\ep_{\rho_{\bf n},\bn}  \in C\left(\left[0,T\right];L^{2}\left(\Omega\right)\right)\cap L^{2}\left(0,T;H_0^{1}\left(\Omega\right)\right)$.
						Assume that Problem \eqref{problem555} has unique solution ${\bf u} \in L^\infty\lf(0,T; 		 \mathcal{Z}_{\gamma, TA_1 }(\Omega)\rt) $.

						\begin{enumerate}[{\bf  \upshape(a)}]
							\item  \text{ \bf Error estimate in $L^2$}.
							Let $H, G$ be as in Theorem \ref{theorem2.1}. 	Let  $\bn, \rho_{\bf n}$ be such that
							\begin{equation}
							\lim_{ |{\bf n}| \to +\infty } \beta_{\bf n}=	\lim_{ |{\bf n}| \to +\infty } \rho_{\bf n}=+\infty , \quad 	\lim_{ |{\bf n}| \to +\infty } \frac{  e^{2\rho_{\bf n} T}   \beta_{\bf n}^{d/2} }{\prod_{k=1}^d  ( n_k)^{4\mu_k} }=\lim_{ |{\bf n}| \to +\infty } e^{2\rho_{\bf n} T}   \bn^{-\mu_0}= 0,
							\end{equation}
							and
							\begin{equation}
							\rho_{\bf n} \le \frac{1}{T} \log \Big (\frac{1}{  {\mathcal E(\ep) }} \Big)
							\end{equation}
							for any  $ 0< {\mathcal E(\ep) }<1$ that  satisfies
							\begin{equation}
							\lim_{\ep \to 0}  \frac{\ep}{ {\mathcal E(\ep) } }=0.
							\end{equation}
							Let $Q_{\bf n}$ be  such that
							\begin{align}
							\lim_{ |{\bf n}| \to +\infty }	e^{6Q_{\bf n}^2T}  \rho_{\bf n}^{-2\gamma}= 		\lim_{ |{\bf n}| \to +\infty } e^{6Q_{\bf n}^2 T}\frac{  e^{2\rho_{\bf n} T}   \beta_{\bf n}^{d/2} }{\prod_{k=1}^d  ( n_k)^{4\mu_k} }=\lim_{ |{\bf n}| \to +\infty }	e^{6Q_{\bf n}^2T}e^{2\rho_{\bf n} T}   \bn^{-\mu_0}  =0.
							\end{align}
							and
							\begin{equation}
							Q_{\bf n} \le \sqrt{\frac{1}{6T} \log \Big (\frac{1}{  {\mathcal E_0(\ep) }} \Big) }
							\end{equation}
							where $ 0< {\mathcal E_0(\ep) }<1$ that  satisfies
							\begin{equation}
							\lim_{\ep \to 0}  \frac{\ep}{ {\mathcal E(\ep) } {\mathcal E_0(\ep) } }=0.
							\end{equation}
							Then for $|{\bf n}|$ large enough, and $\ep$ small enough
							$
							{\bf E}	\Big\| \widehat U^\ep_{\rho_{\bf n},\bn}-u \Big\|_{L^2(\Omega)}
							$
							is of order
							\begin{equation}
							e^{6Q_{\bf n}^2T}  \max \Bigg( 	  \frac{  e^{2\rho_{\bf n} (T-t)}   \beta_{\bf n}^{d/2} }{\prod_{k=1}^d  ( n_k)^{4\mu_k} },   e^{-2\rho_{\bf n} t} \rho_{\bf n}^{-2\gamma}, e^{2\rho_{\bf n} (T-t)}  \bn^{-\mu_0} \Bigg)+  \frac{\ep}{ {\mathcal E(\ep) } {\mathcal E_0(\ep) } }  . \label{er444}
							\end{equation}

\item  \text{ \bf Error estimate in $H^{1}(\Omega)$}.
							Let  $G$ be as in Theorem \ref{theorem2.1} and $H \in \mathcal{H}^{\mu_0+1}(\Omega)$. Let $\bn, \rho_{\bf n}$  be such that
							\begin{equation}
							\lim_{ |{\bf n}| \to +\infty } \beta_{\bf n}=	\lim_{ |{\bf n}| \to +\infty } \rho_{\bf n}=+\infty , \quad 	\lim_{ |{\bf n}| \to +\infty } \frac{  e^{2\rho_{\bf n} T}   \beta_{\bf n}^{\frac{d+2}{2}} }{\prod_{k=1}^d  ( n_k)^{4\mu_k} }=\lim_{ |{\bf n}| \to +\infty } e^{2\rho_{\bf n} T}   \bn^{-\mu_0}=  0,
							\end{equation}
							and
							\begin{equation}
							\rho_{\bf n} \le \frac{1}{T} \log \Big (\frac{1}{  {\mathcal E(\ep) }} \Big).
							\end{equation}
							Assume that  $Q_{\bf n}$  satisfies  that
							\begin{align}
							\lim_{ |{\bf n}| \to +\infty } \exp \Big(  \frac{48T Q_{\bf n}^2}{A_1-A_0}  \Big)  \rho_{\bf n}^{-2\gamma}
							=	\lim_{ |{\bf n}| \to +\infty }  \exp \Big(  \frac{48T Q_{\bf n}^2}{A_1-A_0}  \Big)  \frac{  e^{2\rho_{\bf n} T}   \beta_{\bf n}^{\frac{d+2}{2}} }{\prod_{k=1}^d  ( n_k)^{4\mu_k} }=0
							\end{align}
							and
							\begin{equation}
							Q_{\bf n} \le \sqrt{\frac{A_1-A_0}{48T} \log \Big (\frac{1}{  {\mathcal E_0(\ep) }} \Big) }.
							\end{equation}
							Then for $|{\bf n}|$ large enough, and $\ep$ small enough
							$
							{\bf E}	\Big\| \widehat U^\ep_{\rho_{\bf n},\bn}-{\bf u} \Big\|_{H^{1}(\Omega)}
							$
							is of order
							\begin{equation}
							\exp \Big(  \frac{48T Q_{\bf n}^2}{A_1-A_0}  \Big)\max \Bigg( 	  \frac{  e^{2\rho_{\bf n} (T-t)}   \beta_{\bf n}^{\frac{d+2}{2}} }{\prod_{k=1}^d  ( n_k)^{4\mu_k} },   e^{-2\rho_{\bf n} t} \rho_{\bf n}^{-2\gamma}, e^{2\rho_{\bf n} (T-t)}  \bn^{-\mu_0} \Bigg)+  \frac{\ep}{ {\mathcal E(\ep) } {\mathcal E_0(\ep) } } . \label{er555}
							\end{equation}
						\end{enumerate}
					\end{theorem}
					
						\begin{remark}  \label{lemma3.1}
						Let us give one choice for  $ \beta_{\bf n}$ as follows
							\begin{equation} \label{betan}
							\beta_{\bf n} = \left( \prod_{k=1}^d  n_k  \right)^{\frac{1}{ 2 \al_0+d/2 }}
							\end{equation}
							then we choose $ \rho_{\bf n}$ such that
							\begin{equation}
						\rho_{\bf n}= \frac{\al_0 }{T (2 \al_0+d/2)} \log \left(  \prod_{k=1}^d  n_k \right).
							\end{equation}
					Since $	\lim_{\ep \to 0}  \frac{\ep}{ {\mathcal E(\ep) } }=0,$ we can choose ${\mathcal E(\ep) }=\ep^{m_0} $ for any $0<m_0<1$.  Since 	$\rho_{\bf n} \le \frac{1}{T} \log \Big (\frac{1}{  {\mathcal E(\ep) }} \Big)$, we know that
					\[
					 \frac{\al_0 }{T (2 \al_0+d/2)} \log \left(  \prod_{k=1}^d  n_k \right) \le \frac{m_0}{T} \log(\frac{1}{\ep}) .  \]
				Let us choose  $Q_{\bf n}$ such that  $  e^{6 T Q_{\bf n}^2 }= \left(\overline \Pi({\bf n}) \right)^{\delta_0-1}  $
					for any $0<  \delta_0 <1$.  So, we have
					\begin{equation}
				Q_{\bf n}= \sqrt{\frac{ \delta_0-1 }{6T} \log \left( \overline \Pi({\bf n}) \right)}.
					\end{equation}
					Here $\overline \Pi({\bf n}) $ is defined   by
						\begin{equation}\label{def:pi-n}
						\overline \Pi({\bf n}) = \max \Bigg(  e^{2T{\overline  M} (\sqrt{\rho_{\bf n}})}    \beta_{\bf n}^{d/2} \prod_{k=1}^d  ( n_k)^{-4\mu_k}, e^{2T{\overline  M } (\sqrt{\rho_{\bf n}})  } \beta_{\bf n}^{-\mu_0} ,  \Big| {\overline  M}(|{\sqrt{\rho_{\bf n}}}|)\Big|^{-2\gamma} \Bigg).
						\end{equation}
					 Let us continue to choose ${\mathcal E_0(\ep) }= \ep^{m_1}$ for any $0<m_1<1-m_0$. Since 	$Q_{\bf n} \le \sqrt{\frac{1}{6T} \log \Big (\frac{1}{  {\mathcal E_0(\ep) }} \Big) } $, we have
					\[
					\log \left( \frac{1}{ \overline \Pi({\bf n})} \right)  \le \frac{m_1}{1-\delta_0} 	\log \left( \frac{1}{\ep} \right).
					\]

					Using similar argument as above, we can choose $\bn, \rho_{\bf n}, Q_{\bf n}$ for Part (b) of Theorem \eqref{theorem5.1}.
 					\end{remark}
 					
 					\begin{remark}
 						Our  analysis and techniques in this section  can be applied to consider a  space fractional version   of the   {\bf Ginzburg-Landau } type equation
 						\begin{equation}
 						\label{problem5555}
 						\left\{\begin{array}{l l l}
 						{\bf u}_t + \Lambda(t) (-\Delta )^\beta 	{\bf u} & =B(	{\bf x},t) 	{\bf u}-C(	{\bf x},t)	{\bf u}^3+G(	{\bf x},t), & \qquad (x,t) \in \Omega\times (0,T),\\
 						{\bf u}(	{\bf x},t)&=0, & \qquad 	{\bf x} \in \partial {\Omega},\\
 						{\bf u}(	{\bf x},T) & = H(	{\bf x}), & \qquad 	{\bf x} \in {\Omega},
 						\end{array}
 						\right.
 						\end{equation}
 						{ 	where the fractional Laplacian  $(-\Delta )^\beta  $  is defined by the spectral theorem as follows for nice functions
 							$$
 							(-\Delta )^\beta h({\bf x})=\sum_{{\bf p}\in \mathbb{N}^d} \lambda_{\bf p}^\beta <h, {\bf \psi}_{\bf p}>{\bf \psi}_{\bf p}({\bf x})=\sum_{{\bf p}\in \mathbb{N}^d} |{\bf p}|^{2\beta} <h, {\bf \psi}_{\bf p}>{\bf \psi}_{\bf p}({\bf x}).
 							$$
 							See \cite{Tuan2} for more about this operator. }
 						Here, $B$ and $C$ are randomly perturbed observations
 						\begin{equation*}
 						\widetilde B_{i_1,i_2,...i_d}(t):=B	( x_{i_1},...x_{i_d},t)+\vartheta  \Psi_{i_1,i_2,...i_d}(t), \quad i_k= \overline {1,n_k},\quad k=\overline {1,d}
 						\end{equation*}
 						and
 						\begin{equation*}
 						\widetilde C_{i_1,i_2,...i_d}(t):=C	( x_{i_1},...x_{i_d},t)+\vartheta  \Psi_{i_1,i_2,...i_d}(t), \quad i_k= \overline {1,n_k},\quad k=\overline {1,d}.
 						\end{equation*}
 						First, thanks to Theorem \ref{theorem2.1},  we define the following functions
 						\begin{equation*}
 						\widehat B_{\beta_{\bf n}}({\bf x},t) = \sum_{{\bf p} \in  	\mathcal{W}_{  \beta_{\bf n}}   } \Bigg[\frac{\pi^d}{\prod_{k=1}^d n_k } \sum_{i_1=1}^{n_1}...\sum_{i_d=1}^{n_d} \widetilde B_{i_1,i_2,...i_d,t} \psi_p	( x_{i_1},...x_{i_d}) \Bigg] \psi_{\bf p}({\bf x})
 						\end{equation*}
 						and
 						\begin{equation*}
 						\widehat C_{\beta_{\bf n}}({\bf x},t) = \sum_{{\bf p} \in  	\mathcal{W}_{  \beta_{\bf n}}   } \Bigg[\frac{\pi^d}{\prod_{k=1}^d n_k } \sum_{i_1=1}^{n_1}...\sum_{i_d=1}^{n_d} \widetilde C_{i_1,i_2,...i_d,t} \psi_p	( x_{i_1},...x_{i_d}) \Bigg] \psi_{\bf p}({\bf x}).
 						\end{equation*}
 						Then  we continue to approximate the function $F({\bf v})=B({\bf x},t){\bf v}-C({\bf x},t){\bf v}^3$ by $\overline F_{Q_{\bf n}}$ defined by
 						\[
 						\overline F_{Q_{\bf n}}\left({\bf v}({\bf x},t) \right)
 						=
 						\begin{cases}
 						\widehat B_{\beta_{\bf n}}({\bf x},t) Q_{\bf n}-\widehat C_{\beta_{\bf n}}({\bf x},t)Q_{\bf n}^3, &\quad {\bf v}({\bf x},t) >Q_{\bf n},\\
 						\widehat B_{\beta_{\bf n}}({\bf x},t) {\bf v}-\widehat C_{\beta_n}(x,t) {\bf v}^3, &\quad - Q_{\bf n} \le {\bf v}({\bf x},t)\le Q_{\bf n}, \\
 						-\widehat B_{\beta_{\bf n}}({\bf x},t) Q_{\bf n}+\widehat C_{\beta_{\bf n}}({\bf x},t)Q_{\bf n}^3,&\quad {\bf v}({\bf x},t) < - Q_{\bf n}.
 						\end{cases}
 						\]
 						where we recall  $Q_{\bf n}$ as above.
 						Let $P= -A_1 (-\Delta )^\beta $ and
 						where $	{\bf P}_{\rho_{\bf n}}$ is defined by
 						\begin{equation*}
 						{\bf P}_{\rho_{\bf n}}(v)  = \sum_{{\bf |p|} \le M_{\bf n} A_1^{-\frac{1}{2\beta}  }}^\infty     |{\bf p}|^{2\beta} \big\langle v(x),\phi_p(x)\big\rangle_{L^2(\Omega)} \phi_p(x),
 						\end{equation*}
 						for any function $v \in L^2(\Omega)$.
 						We  can study  	a  regularized solution $\widehat U^\ep_{\rho_{\bf n},\bn}$   which  satisfies
 						\bq
 						\left\{ \begin{gathered}
 						\frac{\partial \widehat U^\ep_{\rho_{\bf n},\bn}}{\partial t}+ \Lambda_\ep(t) (-\Delta )^\beta\widehat U^\ep_{\rho_{\bf n},\bn} -{\bf P}\widehat U^\ep_{\rho_{\bf n},\bn} +{\bf P}_{\rho_{\bf n}}\widehat U^\ep_{\rho_{\bf n},\bn}  = F_{Q_{\bf n}}({\widehat U^\ep_{\rho_{\bf n},\bn} }({\bf x},t))+\widehat G_{\rho_{\bf n}}({\bf x},t) ,~~0<t<T, \hfill \\
 						{\widehat U^\ep_{\rho_{\bf n},\bn} }({\bf x},t)= 0,~~{\bf x} \in \partial \Omega,\hfill\\
 						{\widehat U^\ep_{\rho_{\bf n},\bn} }({\bf x},T)=\widehat H_{\rho_{\bf n}}({\bf x}). \hfill\\
 						\end{gathered}  \right. \nn
 						\eq
 						The convergence is not mentioned here and can be similarly worked out  as in  Theorem \eqref{theorem5.1}. This is also an  interesting topic in our  forthcoming works.

 					\end{remark}

					\begin{proof}[\bf Proof of Theorem \ref{theorem5.1}]
						We first prove the existence and uniqueness of the solution.
							Let $ \widehat V^\ep_{\rho_{\bf n},\bn}({\bf x}, t) =  \widehat U^\ep_{\rho_{\bf n},\bn}({\bf x}, T-t)$
							and set $ \overline \Lambda_\ep(t)= A_1- \Lambda_\ep(t) $.
							Then by \eqref{regu333}, it is obvious that  $\overline V_{\rho_{\bf n}, \bn}({\bf x},t) $ satisfies the following equation
						\bq
						\left\{ \begin{gathered}
						\frac{\partial \widehat V^\ep_{\rho_{\bf n}, \bn}}{\partial t} -\overline \Lambda_\ep(t) \Delta \widehat V^\ep_{\rho_{\bf n}, \bn} = 	\mathcal G \widehat V^\ep_{\rho_{\bf n}, \bn} ({\bf x},t))-\widehat G_{\beta_{\bf n} }({\bf x},t) ,~~0<t<T, \hfill \\
						{\widehat V^\ep_{\rho_{\bf n},\bn} }({\bf x},t)= 0,~~{\bf x} \in \partial \Omega,\hfill\\
						{\widehat V^\ep_{\rho_{\bf n},\bn} }({\bf x},0)=\widehat H_{\beta_{\bf n}}({\bf x}), \hfill\\
						\end{gathered}  \right. \label{regu33333}
						\eq
						where $\mathcal G$ is given by
						\[
						\mathcal G  w({\bf x},t)  = {\bf P}_{\rho_{\bf n}}w ({\bf x},t)  - F_{Q_{\bf n}} (w({\bf x},t)),
						\]
							for any  $w \in C\lf([0,T];L^2(\Omega)\rt)$.
						For any $w_1, w_2 \in C\lf([0,T];L^2(\Omega)\rt) $,  we obtain the following estimate
						\begin{align}
						&\|	\mathcal G  w_1 ({\bf x},t) - 	\mathcal G  w_2({\bf x},t)  \|_{L^2(\Omega)} \nn\\
						&\quad \quad \quad \quad \quad \le \|  F_{Q_{\bf n}} (w_1({\bf x},t))- F_{Q_{\bf n}} (w_2({\bf x},t)) \|_{L^2(\Omega)}  +\Big\|  {\bf P}_{\rho_{\bf n}}w_1 ({\bf x},t) -{\bf P}_{\rho_{\bf n}}w_2 ({\bf x},t)    \Big\|_{L^2(\Omega)} \nn\\
						&\quad \quad \quad \quad \quad \le \lf[ \left(2+6Q_{\bf n}^2 \right) + \rho_{\bf n} \rt]~\|w_1(\cdot,t)-w_2(\cdot,t) \|_{L^2(\Omega)} .
						\end{align}
						which we used  \eqref{local3} and \eqref{estimateP1}.
						So $\mathcal{G}$ is a Lipschitz function. Using the results of Chapter 12, Theorem 12.2, page 211 of \cite{chipot}, we complete the proof of Step 1. \\	
					For estimates between $\widehat U^\ep_{\rho_{\bf n},\bn}$ and ${\bf u}$, we divide the proof into two Parts:\\
						{\bf Part 1.} {\it  Error estimate in $L^2$: }\\
						Let $\overline \Lambda_\ep(t)= A_1- \Lambda_\ep(t)$.
						The main equation in \eqref{problem555} can be rewritten as follows
						\begin{align*}
						\frac{\partial {\bf u} }{\partial t} -  \Lambda_\ep(t) \Delta  {\bf u}& = F({\bf u}({\bf x},t))+ G({\bf x},t)- \Big(\Lambda_\ep(t)-  \Lambda(t) \Big)  \Delta   {\bf u}.
						\end{align*}
						For $\nu_{\bf n}>0$, we put $$ \mathbf{ Y}^\ep_{\rho_{\bf n},\nu_{\bf n}, \bn}({\bf x},t)=e^{\nu_{\bf n}(t-T)}\Big[ \widehat U^\ep_{\rho_{\bf n},\bn}({\bf x},t)-\textbf{u}({\bf x},t)\Big].$$
Then, from the last two  equalities,  and  an elementary computation gives
						\begin{equation}  \label{est1}
						\begin{split}
						\frac{\partial \mathbf{ Y}^\ep_{\rho_{\bf n},\nu_{\bf n}, \bn}({\bf x},t)}{\partial t}  &+  \overline \Lambda_\ep(t) \Delta  \mathbf{ Y}^\ep_{\rho_{\bf n},\nu_{\bf n},\bn}({\bf x},t)-\nu_{\bf n} \mathbf{ Y}^\ep_{\rho_{\bf n},\nu_{\bf n},\bn }({\bf x},t)\\
						& = -e^{\nu_{\bf n}(t-T)}  {\bf P}_{\rho_{\bf n}} \mathbf{ Y}^\ep_{\rho_{\bf n},\nu_{\bf n},\bn}({\bf x},t)+e^{\nu_{\bf n}(t-T)} \left( {\bf P}_{\rho_{\bf n}}-{\bf P} \right)\textbf{u}({\bf x},t) \\
						&  + e^{\nu_{\bf n}(t-T)}\Big(\Lambda_\ep(t)-  \Lambda(t)  \Big) \Delta  {\bf u}({\bf x},t) \\
						&+ e^{\nu_{\bf n} (t-T)}\Big[F_{Q_{\bf n} }({\widehat U^\ep_{\rho_{\bf n}, \bn} }({\bf x},t))- F({\bf u}({\bf x},t))\Big] + e^{\nu_{\bf n} (t-T)}\left[\widehat G_{\beta_{\bf n}}({\bf x},t)- G({\bf x},t)\right], 
						\end{split}
						\end{equation}
						and $$\mathbf{ Y}^\ep_{\rho_{\bf n},\nu_{\bf n}, \bn}({\bf x},t)|_{\partial \Omega}=0, \quad  \mathbf{ Y}^\ep_{\rho_{\bf n},\nu_{\bf n},\bn}({\bf x},T)=\widehat H_{\beta_{\bf n}}({\bf x})-H({\bf x}). $$
						By taking the inner product two sides of \eqref{est1} with $\mathbf{ Y}_{\rho_{\bf n},\nu_{\bf n}}$ and integrating over  $(t,T)$
						one deduces that
						\begin{align} \label{3J}
						& \|\mathbf{Y}_{\rho_{\bf n},\nu_{\bf n}, \bn}(\cdot, T)\|^2_{L^2(\Omega)}-\|\mathbf{Y}_{\rho_{\bf n},\nu_{\bf n},\bn}(\cdot, t)\|^2_{L^2(\Omega)} \nn\\
						&- 2\int_t^T \overline \Lambda_\ep(\tau)  \|\mathbf{Y}^\ep_{\rho_{\bf n},\nu_{\bf n}, \bn}(\cdot, \tau)\|^2_{H^{1}(\Omega)}d\tau - 2\nu_{\bf n} \int_t^T \|\mathbf{Y}^\ep_{\rho_{\bf n},\nu_{\bf n}, \bn}(\cdot, \tau)\|^2_{L^2(\Omega)} d\tau\nn\\
						&=\mathcal{L}_{1,{\bf n}}(t)+\mathcal{L}_{2,{\bf n}}(t)+\mathcal{L}_{3,{\bf n}}(t)+\mathcal{L}_{4,{\bf n}}(t)+\mathcal{L}_{5,{\bf n}}(t),
						\end{align}
						where
						\begin{equation}
						\mathcal{L}_{1,{\bf n}}(t):= 	-2 \int_t^T \int_{\Omega} e^{\nu_{\bf n}(\tau-T)}  {\bf P}_{\rho_{\bf n}} \mathbf{ Y}^\ep_{\rho_{\bf n},\nu_{\bf n},\bn}({\bf x},\tau)  \mathbf{ Y}^\ep_{\rho_{\bf n},\nu_{\bf n},\bn}({\bf x},\tau) d{\bf x} d\tau
						\end{equation}
						\begin{equation}
						\mathcal{L}_{2,{\bf n}}(t):= 	2 \int_t^T \int_{\Omega} e^{\nu_{\bf n}(\tau-T)} \left( {\bf P}_{\rho_{\bf n}}-{\bf P} \right)\textbf{u}({\bf x},\tau)  \mathbf{ Y}^\ep_{\rho_{\bf n},\nu_{\bf n},\bn}({\bf x},\tau) d{\bf x} d\tau
						\end{equation}
						\begin{equation}
						\mathcal{L}_{3,{\bf n}}(t):= -	2 \int_t^T \int_{\Omega} e^{\nu_{\bf n}(\tau-T)} \Big(\Lambda_\ep(\tau)-  \Lambda(\tau)  \Big)\Delta \textbf{u}({\bf x},\tau)  \mathbf{ Y}^\ep_{\rho_{\bf n},\nu_{\bf n},\bn}({\bf x},\tau) d{\bf x} d\tau
						\end{equation}
						\begin{equation}
						\mathcal{L}_{4,{\bf n}}(t):= 	2 \int_t^T \int_{\Omega} e^{\nu_{\bf n}(\tau-T)} \Big[F_{Q_{\bf n}}({\widehat U^\ep_{\rho_{\bf n},\bn} }({\bf x},\tau))- F({\bf u}({\bf x},\tau))\Big] \mathbf{ Y}^\ep_{\rho_{\bf n},\nu_{\bf n},\bn}({\bf x},\tau) d{\bf x} d\tau
						\end{equation}
						\begin{equation}
						\mathcal{L}_{5,{\bf n}}(t):= 	2 \int_t^T \int_{\Omega} e^{\nu_{\bf n}(\tau-T)} \Big[\widehat G_{\beta_{\bf n}}({\bf x},\tau)- G({\bf x},\tau)\Big] \mathbf{ Y}^\ep_{\rho_{\bf n},\nu_{\bf n},\bn}({\bf x},\tau) d{\bf x} d\tau.
						\end{equation}
						The expectation of $|\mathcal{L}_{1,{\bf n}}(t)|$	is bounded by
						\begin{align} \label{est2}
						{\bf E}	|\mathcal{L}_{1,{\bf n}}(t)|  &\le 2{\bf E}	 \Bigg(\int_t^T \| {\bf P}_{\rho_{\bf n}} \mathbf{ Y}^\ep_{\rho_{\bf n},\nu_{\bf n},\bn}(.,\tau)\|_{L^2(\Omega)}  \| \mathbf{ Y}^\ep_{\rho_{\bf n},\nu_{\bf n},\bn}(.,\tau) \|_{L^2(\Omega)} d\tau \Bigg) \nn\\
						&\le  2 \rho_{\bf n}  \int_t^T  {\bf E} \| \mathbf{ Y}^\ep_{\rho_{\bf n},\nu_{\bf n},\bn}(.,\tau) \|^2_{L^2(\Omega)} d\tau
						\end{align}		
						where we have used  \eqref{estimateP1}. 		
						For the term 	$\mathcal{L}_{2,{\bf n}}(t)$, it follows by \eqref{lemma4.1} and Cauchy-Schwartz inequality  that
						\begin{align} \label{est3}
						{\bf E}	|\mathcal{L}_{2,{\bf n}}(t)| &\le  	{\bf E} \Bigg( \int_t^T e^{2\nu_{\bf n}(\tau-T)}  A_1^2 \rho_{\bf n}^{-2\gamma} e^{-2T\rho_{\bf n} } \|\textbf{u}\|_{ L^\infty\lf(0,T; 		 \mathcal{Z}_{\gamma, TA_1 }(\Omega)\rt) }^2  d\tau \Bigg)	\nn\\
						&+{\bf E} \Bigg(  \int_t^T \| \mathbf{ Y}^\ep_{\rho_{\bf n},\nu_{\bf n},\bn}(.,\tau) \|_{L^2(\Omega)}^2 d\tau \Bigg)\nn\\
						&\le TA_1^2 \rho_{\bf n}^{-2\gamma} e^{-2T\rho_{\bf n} }	\|\textbf{u}\|_{L^\infty\lf(0,T; 		 \mathcal{Z}_{\gamma, TA_1 }(\Omega)\rt)}^2  + \int_t^T  {\bf E} \| \mathbf{ Y}^\ep_{\rho_{\bf n},\nu_{\bf n},\bn}(.,\tau) \|^2_{L^2(\Omega)} d\tau.
						\end{align}
						Next, for $\mathcal{L}_{3,{\bf n}}(t)$, 	From the inequality $2\langle a_1, a_2\rangle_{L^2(\Omega)} \leq \|a_1\|^2_{L^2(\Omega)} + \|a_2\|^2_{L^2(\Omega)}$ for any  $a_i \in L^2(\Omega) , ~(i=1,2)$, we infer
						\begin{equation} \label{est4}
						\begin{split}
						{\bf E}	|\mathcal{L}_{3,{\bf n}}(t)| &\le  	{\bf E} \Bigg( \int_t^T e^{2\nu_{\bf n} (\tau-T)}  \Big(\Lambda_\ep(\tau)-  \Lambda(\tau)  \Big)^2   \lf\| \Delta   {\bf u} \rt\|_{L^2(\Omega)}^2 d\tau  \Bigg)	+  	{\bf E} \Bigg(  \int_t^T \| \mathbf{Y}^\ep_{\rho_{\bf n},\nu_{\bf n},\bn}(., \tau)\|^2_{L^2(\Omega)} d\tau  \Bigg)	\\
						&\le   \|{\bf u}\|^2_{L^\infty (0,T;H^{2{}}(\Omega))} \int_t^T  {\bf E}  \Big(\Lambda_\ep(\tau)-  \Lambda(\tau)  \Big)^2  d\tau  +\int_t^T  {\bf E} \| \mathbf{ Y}^\ep_{\rho_{\bf n},\nu_{\bf n},\bn}(.,\tau) \|_{L^2(\Omega)}^2 d\tau\\
						&\le \ep^2 \left(\int_0^T \mathbb{\bf E} \Big|\overline  {\xi}(\tau)\Big|^2 d\tau \right) \|{\bf u}\|^2_{L^\infty (0,T;H^{2{}}(\Omega))}+ \int_t^T  {\bf E} \| \mathbf{ Y}^\ep_{\rho_{\bf n},\nu_{\bf n},\bn}(.,\tau) \|_{L^2(\Omega)}^2 d\tau\\
						&\le \ep^2 T^2  \|{\bf u}\|^2_{L^\infty (0,T;H^{2{}}(\Omega))}  + \int_t^T  {\bf E} \| \mathbf{ Y}^\ep_{\rho_{\bf n},\nu_{\bf n},\bn}(.,\tau) \|_{L^2(\Omega)}^2 d\tau,
						\end{split}	
						\end{equation}	
						where  we have used the fact that
						$
						\Lambda_\ep(t)- \Lambda(t)= \ep \overline  {\xi}(t)
						$
						and $\mathbb{\bf E} \Big|\overline  {\xi}(t)\Big|^2 =t$.\\				
For $\mathcal{L}_{4,{\bf n}}$,  thanks to \eqref{local2} and \eqref{local3}, we deduce that
						\begin{align} \label{est5}
						{\bf E}	|\mathcal{L}_{4,{\bf n}}(t)|&\leq	{\bf E} \Bigg(\int_t^T  \lf\|e^{\nu_n (\tau-T)}\Big[F_{Q_{\bf n}}({\widehat U^\ep_{\rho_{\bf n},\bn} }(.,\tau))- F({\bf u}(.,\tau))\Big]\rt\|_{L^2(\Omega)}~\|\mathbf{Y}^\ep_{\rho_{\bf n},\nu_{\bf n},\bn}(., \tau)\|_{L^2(\Omega)}d\tau \Bigg) \nn\\
						&\leq  2 \left( 1+3Q_n^2\right)\int_t^T 	{\bf E} \|\mathbf{Y}^\ep_{\rho_{\bf n},\nu_{\bf n},\bn}(., \tau)\|_{L^2(\Omega)}^2 d\tau.
						\end{align}
						The term ${\bf E}	|\mathcal{L}_{5,{\bf n}}(t)|$ can be bounded by
						\begin{equation} \label{est6}
\begin{split}
						{\bf E}	|\mathcal{L}_{5,{\bf n}}(t)|& \le	{\bf E} \Bigg( \int_t^T e^{2\nu_{\bf n} (\tau-T)} \lf\|\widehat G_{\rho_{\bf n}}(.,\tau)- G(.,\tau)\rt\|^2_{L^2(\Omega)}d\tau+  \int_t^T \|\mathbf{Y}_{\rho_{\bf n},\nu_{\bf n},\bn}(., \tau)\|^2_{L^2(\Omega)} d\tau\Bigg)\\
						&\le {\bf E} T \lf\|\widehat G_{\beta_{\bf n}}- G\rt\|^2_{L^\infty(0,T;L^2(\Omega))}+\int_t^T {\bf E} \| \mathbf{ Y}^\ep_{\rho_{\bf n},\nu_{\bf n}, \bn}(.,\tau) \|_{L^2(\Omega)}^2d\tau  .
\end{split}						
\end{equation}
						Combining \eqref{est1}, \eqref{est2}, \eqref{est3}, \eqref{est4}, \eqref{est5}, \eqref{est6}, we conclude that
						\begin{eqnarray}
						\begin{aligned} \label{3J}
						&{\bf E}\|\mathbf{Y}^\ep_{\rho_{\bf n},\nu_{\bf n}, \bn}(\cdot, T)\|^2_{L^2(\Omega)}-{\bf E}\|\mathbf{Y}^\ep_{\rho_{\bf n},\nu_{\bf n}, \bn}(\cdot, t)\|^2_{L^2(\Omega)} \nn\\
						&- 2\int_t^T \overline \Lambda_\ep(\tau)   {\bf E} \|\mathbf{Y}_{\rho_{\bf n},\nu_{\bf n}, \bn}(\cdot, \tau)\|^2_{H^{1}(\Omega)}d\tau -2 \nu_n \int_t^T {\bf E}\|\mathbf{Y}^\ep_{\rho_{\bf n},\nu_{\bf n}, \bn}(\cdot, \tau)\|^2_{L^2(\Omega)} d\tau\nn\\
						&\ge -2 \rho_{\bf n}  \int_t^T  {\bf E} \| \mathbf{ Y}^\ep_{\rho_{\bf n},\nu_{\bf n}, \bn}(.,\tau) \|^2_{L^2(\Omega)} d\tau-T \rho_{\bf n}^{-2\gamma }  e^{-2T\rho_{\bf n} }	{\bf E}  \|\textbf{u}\|_{\kg}^2 \nn\\
						&-  T {\bf E} \lf\|\widehat G_{\beta_{\bf n}}- G \rt\|^2_{L^\infty(0,T;L^2(\Omega))} -\ep^2   (T-t) \|{\bf u}\|^2_{L^\infty (0,T);H^{2}(\Omega))}\nn\\
						&-\left( 5+6Q_{\bf n}^2\right)\int_t^T 	{\bf E} \|\mathbf{Y}^\ep_{\rho_{\bf n},\nu_{\bf n}, \bn}(., \tau)\|_{L^2(\Omega)}^2 d\tau.
						\end{aligned}
						\end{eqnarray}
						This implies that
						\begin{align}
						{\bf E}\|\mathbf{Y}^\ep_{\rho_{\bf n},\nu_{\bf n}, \bn }(\cdot, t)\|^2_{L^2(\Omega)}  &+ \left(2 \nu_{\bf n} -2\rho_{\bf n} -5-6Q_{\bf n}^2 \right)\int_t^T 	{\bf E} \|\mathbf{Y}^\ep_{\rho_{\bf n},\nu_{\bf n}, \bn}(., \tau)\|_{L^2(\Omega)}^2 d\tau \nn\\
						&\le {\bf E}\|\mathbf{Y}^\ep_{\rho_{\bf n},\nu_{\bf n}, \bn}(\cdot, T)\|^2_{L^2(\Omega)}+ T \rho_n^{-2\gamma }  e^{-2T\rho_{\bf n} }	{\bf E}  \|\textbf{u}\|_{\kg}^2 \nn\\
						&+\ep^2   (T-t) \|{\bf u}\|^2_{L^\infty (0,T);H^{2{}}(\Omega))}+  T {\bf E} \lf\|\widehat G_{\beta_{\bf n}}- G \rt\|^2_{L^\infty(0,T;L^2(\Omega))} .
						\end{align}
						Let us choose $\nu_{\bf n}=\rho_{\bf n}$. Then	
						\begin{align}
						&e^{2\nu_{\bf n}(t-T)}	{\bf E}\|\widehat U^\ep_{\rho_{\bf n},\bn}(.,t)-\textbf{u}(.,t)\|^2_{L^2(\Omega)}\nn\\
						&\quad \quad \le  \left(5+6Q_{\bf n}^2 \right)\int_t^T 	e^{2\nu_{\bf n}(\tau-T)}	{\bf E}\|\widehat U^\ep_{\rho_{\bf n},\bn}(.,\tau)-\textbf{u}(.,\tau)\|^2_{L^2(\Omega)}   d\tau\nn\\
						&\quad \quad+{\bf E}\|\widehat H_{\beta_{\bf n}}-H\|^2_{L^2(\Omega)}+   T {\bf E} \lf\|\widehat G_{\beta_{\bf n}}- G \rt\|^2_{L^\infty(0,T;L^2(\Omega))}  \nn\\
						&\quad \quad+TA_1^2 \rho_{\bf n}^{-2\gamma} e^{-2T\rho_{\bf n} }	\|\textbf{u}\|_{\kg}^2  +\ep^2   (T-t) \|{\bf u}\|^2_{L^\infty (0,T);H^{2}(\Omega))}.
						\end{align}
						Multiplying both sides by $	e^{2\rho_{\bf n} T}	$ we get
						\begin{eqnarray}
						\begin{aligned}
						&e^{2\rho_{\bf n} t}	{\bf E}\|\widehat U^\ep_{\rho_{\bf n},\bn}(.,t)-\textbf{u}(.,t)\|^2_{L^2(\Omega)}  \nn\\
						&\le  \left(5+6Q_{\bf n}^2 \right)\int_t^T 	e^{2\rho_{\bf n} \tau}	{\bf E}\|\widehat U^\ep_{\rho_{\bf n},\bn}(.,\tau)-\textbf{u}(.,\tau)\|^2_{L^2(\Omega)}   d\tau\nn\\
						&+e^{2\rho_{\bf n} T}  {\bf E}\|\widehat H_{\beta_{\bf n} }-H\|^2_{L^2(\Omega)}+  e^{2\rho_{\bf n} T}   T {\bf E} \lf\|\widehat G_{\beta_{\bf n}}- G \rt\|^2_{L^\infty(0,T;L^2(\Omega))} \nn\\
						&+\ep^2 e^{2\rho_{\bf n} T}  T \|{\bf u}\|^2_{L^\infty (0,T);H^{2{}}(\Omega))}+ TA_1^2 \rho_{\bf n}^{-2\gamma} e^{-2T\rho_{\bf n} }	\|\textbf{u}\|_{\kg}^2. \nn\\
						\end{aligned}
						\end{eqnarray}
						Hence, Gronwall's inequality yields the desired estimate
						\begin{eqnarray}
						\begin{aligned}
						&e^{2\rho_{\bf n} t}	{\bf E}\|\widehat U^\ep_{\rho_{\bf n},\bn}(.,t)-\textbf{u}(.,t)\|^2_{L^2(\Omega)}   \nn\\
						&\quad \quad \quad \le \Bigg[ e^{2\rho_{\bf n} T}  {\bf E}\|\widehat H_{\beta_n}-H\|^2_{L^2(\Omega)}+ T e^{2\rho_{\bf n} T}  {\bf E} \lf\|\widehat G_{\beta_{\bf n}}- G \rt\|^2_{L^\infty(0,T;L^2(\Omega))} \nn\\
						&\quad \quad \quad+  TA_1^2 \rho_{\bf n}^{-2\gamma} 	\|\textbf{u}\|_{\kg}^2+\ep^2 e^{2\rho_{\bf n} T}  T \|{\bf u}\|^2_{L^\infty (0,T);H^{2}(\Omega))} \Bigg] e^{(5+6Q_{\bf n}^2)(T-t)}.
						\end{aligned}
						\end{eqnarray}
					This implies that
						\begin{align}  \label{er44}
						&	{\bf E}\|\widehat U^\ep_{\rho_{\bf n},\bn}(.,t)-\textbf{u}(.,t)\|^2_{L^2(\Omega)}   \nn\\
						&\quad  \le e^{2\rho_{\bf n} (T-t) }\Bigg[   {\bf E}\|\widehat H_{\beta_n}-H\|^2_{L^2(\Omega)}+ T {\bf E} \lf\|\widehat G_{\beta_{\bf n}}- G \rt\|^2_{L^\infty(0,T;L^2(\Omega))} \nn\\
						&\quad \quad \quad+  TA_1^2 \rho_{\bf n}^{-2\gamma}  e^{-2\rho_{\bf n} t}	\|\textbf{u}\|_{\kg}^2+\ep^2 e^{2\rho_{\bf n} (T-t)}  T \|{\bf u}\|^2_{L^\infty (0,T);H^{2}(\Omega))} \Bigg] e^{(5+6Q_{\bf n}^2)(T-t)}.
						\end{align}
						From Corollary \eqref{corollary2.1}, we see that the term
						$ {\bf E}	\Big\| \widehat H_{\beta_{\bf n}}-H \Big\|_{L^2(\Omega)}^2+T {\bf E}	\Big\| \widehat G_{\beta_{\bf n}}-G \Big\|_{L^\infty(0,T;L^2(\Omega))}^2$
						is of order
						\begin{equation}  \label{veryimportant}
						\max \Bigg( \frac{   \beta_{\bf n}^{d/2} }{\prod_{k=1}^d  ( n_k)^{4\mu_k} }, ~\beta_{\bf n}^{-\mu_0}  \Bigg).
						\end{equation}
						This together with \eqref{er44}, we obtain \eqref{er444}.

						\noindent {\bf Part 2. Estimate in $H^1(\Omega)$:}
						Recall that $H^1(\Omega)$   is the space of the function $f$ such that  $f$ and   $f'$ belong to $L^2(\Omega)$
with the norm defined by
$\| f\|_{H^1}^2= \|f \|^2_{L^2}+ \|f ' \|^2_{L^2}$.

						By taking the inner product two sides of \eqref{est1}  with $-\Delta \mathbf{ Y}^\ep_{\rho_{\bf n},\nu_{\bf n},\bn}$, and integrating over $(t, T)$
						one deduces that
						\begin{align} \label{h1}
						& \|\mathbf{Y}^\ep_{\rho_{\bf n},\nu_{\bf n},\bn}(\cdot, T)\|^2_{H^1(\Omega)}-\|\mathbf{Y}^\ep_{\rho_{\bf n},\nu_{\bf n},\bn}(\cdot, t)\|^2_{H^1(\Omega)} \nn\\
						&- 2\int_t^T \overline \Lambda_\ep(\tau)  \|\Delta \mathbf{Y}^\ep_{\rho_{\bf n},\nu_{\bf n},\bn}(\cdot, \tau)\|^2_{L^{2}(\Omega)}d\tau - 2\nu_{\bf n} \int_t^T \|\mathbf{Y}^\ep_{\rho_{\bf n},\nu_{\bf n},\bn}(\cdot, \tau)\|^2_{H^1(\Omega)} d\tau\nn\\
						&=\mathcal{L}_{7,{\bf n}}(t)+\mathcal{L}_{8,{\bf n}}(t)+\mathcal{L}_{9,{\bf n}}(t)+\mathcal{L}_{10,{\bf n}}(t)+\mathcal{L}_{11,{\bf n}}(t)
						\end{align}
						where
						\begin{equation}
						\mathcal{L}_{7,{\bf n}}(t):= 	2 \int_t^T \int_{\Omega} e^{\nu_{\bf n}(\tau-T)}  {\bf P}_{\rho_{\bf n}} \mathbf{ Y}^\ep_{\rho_{\bf n},\nu_{\bf n},\bn}({\bf x},\tau)  \Delta \mathbf{ Y}^\ep_{\rho_{\bf n},\nu_{\bf n},\bn}({\bf x},\tau) d{\bf x} d\tau.
						\end{equation}
						\begin{equation}
						\mathcal{L}_{8,{\bf n}}(t):= -	2 \int_t^T \int_{\Omega} e^{\nu_{\bf n}(\tau-T)} \left( {\bf P}_{\rho_{\bf n}}-{\bf P} \right)\textbf{u}({\bf x},\tau)\Delta   \mathbf{ Y}^\ep_{\rho_{\bf n},\nu_{\bf n},\bn}({\bf x},\tau) d{\bf x} d\tau.
						\end{equation}
						\begin{equation}
						\mathcal{L}_{9,{\bf n}}(t):= 	2 \int_t^T \int_{\Omega} e^{\nu_{\bf n}(\tau-T)} \Big(\Lambda_\ep(\tau)-  \Lambda(\tau)  \Big)\Delta \textbf{u}({\bf x},\tau)  \Delta \mathbf{ Y}^\ep_{\rho_{\bf n},\nu_{\bf n},\bn}({\bf x},\tau) d{\bf x} d\tau.
						\end{equation}
						\begin{equation}
						\mathcal{L}_{10,{\bf n}}(t):= -	2 \int_t^T \int_{\Omega} e^{\nu_{\bf n}(\tau-T)} \Big[F_{Q_{\bf n} }({\widehat U^\epsilon_{\rho_{\bf n},\bn} }({\bf x},\tau))- F({\bf u}({\bf x},\tau))\Big] \Delta \mathbf{ Y}^\ep_{\rho_{\bf n},\nu_{\bf n},\bn}({\bf x},\tau) d{\bf x} d\tau.
						\end{equation}
						\begin{equation}
						\mathcal{L}_{11,{\bf n}}(t):=- 	2 \int_t^T \int_{\Omega} e^{\nu_{\bf n}(\tau-T)} \Big[\widehat G_{\beta_{\bf n}}({\bf x},\tau)- G({\bf x},\tau)\Big] \Delta \mathbf{ Y}^\ep_{\rho_{\bf n},\nu_{\bf n},\bn}({\bf x},\tau) {\bf x} d\tau.
						\end{equation}
						{Using the bound from Lemma \ref{lem3.1} we can easily deduce that }
						\begin{align}
						{\bf E}	|\mathcal{L}_{7,{\bf n}}(t)|
						\le  2 \rho_{\bf n}  \int_t^T  {\bf E} \| \mathbf{ Y}^\ep_{\rho_{\bf n},\nu_{\bf n},\bn}(.,\tau) \|^2_{H^1(\Omega)} d\tau. \label{h2}
						\end{align}			
						{Using Cauchy-Schwartz inequality and Lemma \ref{lem3.1}, we estimate 	$|\mathcal{L}_{8,{\bf n}}(t)|$ as follows	}	
						\begin{equation}\label{h3}
						\begin{split}
						{\bf E}	|\mathcal{L}_{8,{\bf n}}(t)| &\le  	{\bf E} \Bigg( \frac{16}{A_1-A_0} \int_t^T e^{2\nu_{\bf n}(\tau-T)}  A_1^2 \rho_{\bf n}^{-2\gamma} e^{-2T\rho_{\bf n} } \|\textbf{u}\|_{\kg}^2  d\tau\Bigg)\\
						& +	{\bf E} \Bigg(  \frac{A_1-A_0}{4}\int_t^T \|\Delta  \mathbf{ Y}_{\rho_{\bf n},\nu_{\bf n},\bn}({\bf x},\tau) \|_{L^2(\Omega)}^2 d\tau \Bigg)	\\
						&\le \frac{16 TA_1^2}{A_1-A_0}  \rho_{\bf n}^{-2\gamma} e^{-2T\rho_{\bf n} }\|\textbf{u}\|_{\kg}^2  +\frac{A_1-A_0}{4} \int_t^T  {\bf E} \|\Delta  \mathbf{ Y}_{\rho_{\bf n},\nu_{\bf n},\bn}(.,\tau) \|^2_{L^2(\Omega)} d\tau,
						\end{split}
						\end{equation}
						{and similarly using the fact that $\Lambda_\ep(\tau)-  \Lambda(\tau)=\epsilon \xi(t)$ we get}
						\begin{align}
						{\bf E}	|\mathcal{L}_{9,{\bf n}}(t)| &\le  	{\bf E} \Bigg(\frac{16}{A_1-A_0}  \int_t^T e^{2\nu_{\bf n}(\tau-T)}  \Big(\Lambda_\ep(\tau)-  \Lambda(\tau)  \Big)^2   \lf\| \Delta   {\bf u} \rt\|_{L^2(\Omega)}^2 d\tau\Bigg)\nn\\
						& +{\bf E} \Bigg( \frac{A_1-A_0}{4} \int_t^T \| \Delta \mathbf{Y}^\ep_{\rho_{\bf n},\nu_{\bf n},\bn}(,, \tau)\|^2_{L^2(\Omega)} d\tau \Bigg) 	\nn\\
						&\le  \frac{16}{A_1-A_0}   \|{\bf u}\|^2_{L^\infty
							(0,T);H^{2{}}(\Omega))} \int_t^T  {\bf E}  \Big(\Lambda_\ep(\tau)-  \Lambda(\tau)  \Big)^2  d\tau  \nn\\
						&+ \frac{A_1-A_0}{4} \int_t^T  {\bf E} \| \Delta \mathbf{ Y}^\ep_{\rho_{\bf n},\nu_{\bf n},\bn}(.,\tau) \|_{L^2(\Omega)}^2 d\tau\nn\\
						&\le \frac{16}{A_1-A_0}   \ep^2   T^2  \|{\bf u}\|^2_{L^\infty
							(0,T);H^{2{}}(\Omega))} +  \frac{A_1-A_0}{4} \int_t^T  {\bf E} \|\Delta  \mathbf{ Y}^\ep_{\rho_{\bf n},\nu_{\bf n},\bn}(.,\tau) \|_{L^2(\Omega)}^2 d\tau. \label{h4}
						\end{align}	
The term $    {\bf E}    |\mathcal{L}_{10,{\bf n}}(t)|$ is estimated  using the fact that for  $|{\bf n}|$ large enough   $F_{Q_{\bf n}}({\bf u }({\bf x},\tau))= F( {\bf u } ({\bf x},\tau))$
						\begin{eqnarray}
						\begin{aligned}
						{\bf E}	|\mathcal{L}_{10,{\bf n}}(t)| &\le \frac{16}{A_1-A_0}  \int_t^T  e^{2\rho_{\bf n}(\tau-T)} \Big\| F_{Q_{\bf n}}({\widehat U^\ep_{\rho_{\bf n},\bn} }({\bf x},\tau))- F({\bf u}(.,\tau)) \Big\|^2_{L^2(\Omega)} +\frac{A_1-A_0}{4} \Big\|\Delta  \mathbf{Y}^\ep_{\rho_{\bf n},\nu_{\bf n},\bn} \Big\|^2_{L^2(\Omega)}\nn\\
						&\le \frac{16}{A_1-A_0}  \left( 1+3Q_{\bf n}^2\right)\int_t^T 	{\bf E} \|\mathbf{Y}^\ep_{\rho_{\bf n},\nu_{\bf n},\bn}(., \tau)\|_{L^2(\Omega)}^2 d\tau+\frac{A_1-A_0}{4} \Big\|\Delta  \mathbf{Y}^\ep_{\rho_{\bf n},\nu_{\bf n},\bn} \Big\|^2_{L^2(\Omega)}.\label{h5}
						\end{aligned}
						\end{eqnarray}
						The term ${\bf E}	|\mathcal{L}_{11,{\bf n}}(t)|$ can be bounded by
						\begin{align}
						{\bf E}	|\mathcal{L}_{11,{\bf n}}(t)| & \le	{\bf E} \Bigg(\frac{16T}{A_1-A_0}  \int_t^T e^{2\nu_{\bf n} (\tau-T)} \lf\|\widehat G_{\beta_{\bf n}}(.,\tau)- G(.,\tau)\rt\|^2_{L^2(\Omega)}d\tau \Bigg)\nn\\
						& +	{\bf E} \Bigg( \frac{A_1-A_0}{4} \int_t^T \|\Delta \mathbf{Y}^\ep_{\rho_{\bf n},\nu_{\bf n},\bn}(., \tau)\|^2_{L^2(\Omega)} d\tau\Bigg)\nn\\
						&\le \frac{16T}{A_1-A_0} {\bf E} \lf\|\widehat G_{\beta_{\bf n}}- G \rt\|^2_{L^\infty(0,T;L^2(\Omega))}+\frac{A_1-A_0}{4}\int_t^T {\bf E} \| \Delta \mathbf{ Y}^\ep_{\rho_{\bf n},\nu_{\bf n},\bn}(.,\tau) \|_{L^2(\Omega)}^2d\tau. \label{h6}
						\end{align}		
						Combining \eqref{h1}, \eqref{h2}, \eqref{h3}, \eqref{h4}, \eqref{h5} gives
						\begin{eqnarray}
						\begin{aligned} \label{h6}
						&{\bf E}\|\mathbf{Y}^\ep_{\rho_{\bf n},\nu_{\bf n},\bn}(\cdot, T)\|^2_{H^1(\Omega)}-{\bf E}\|\mathbf{Y}^\ep_{\rho_{\bf n},\nu_{\bf n},\bn}(\cdot, t)\|^2_{H^1(\Omega)} \nn\\
						&-  2\int_t^T \overline \Lambda_\ep(\tau)  {\bf E} \|\Delta \mathbf{Y}^\ep_{\rho_{\bf n},\nu_{\bf n},\bn}(\cdot, \tau)\|^2_{L^{2}(\Omega)}d\tau - 2\nu_{\bf n} \int_t^T {\bf E}\|\mathbf{Y}^\ep_{\rho_{\bf n},\nu_{\bf n},\bn}(\cdot, \tau)\|^2_{H^1(\Omega)} d\tau\nn\\
						&\ge - 2 \rho_{\bf n}  \int_t^T  {\bf E} \| \mathbf{ Y}^\ep_{\rho_{\bf n},\nu_{\bf n},\bn}(.,\tau) \|^2_{H^1(\Omega)} d\tau- \frac{16 TA_1^2}{A_1-A_0}  \rho_{\bf n}^{-2\gamma} e^{-2T\rho_{\bf n} }\|\textbf{u}\|_{\kg}^2  \nn\\
						&- \frac{16}{A_1-A_0} {\bf E} \lf\|\widehat G_{\beta_{\bf n}}- G\rt\|^2_{L^\infty(0,T;L^2(\Omega))} -\frac{16}{A_1-A_0}   \ep^2   (T-t)  \|{\bf u}\|^2_{L^\infty
							(0,T);H^{2{}}(\Omega))}\nn\\
						&- \frac{16}{A_1-A_0}  \left( 1+3Q_{\bf n}^2\right) \int_t^T 	{\bf E} \|\mathbf{Y}^\ep_{\rho_{\bf n},\nu_{\bf n},\bn}(., \tau)\|_{L^2(\Omega)}^2 d\tau- (A_1-A_0) \int_t^T {\bf E} \| \Delta \mathbf{ Y}^\ep_{\rho_{\bf n},\nu_{\bf n},\bn}(.,\tau) \|_{L^2(\Omega)}^2d\tau.
						\end{aligned}
						\end{eqnarray}
						Choose $\nu_{\bf n}=\rho_{\bf n}$. Then	
						\begin{align}
						&e^{2\rho_{\bf n}(t-T)}	{\bf E}\|\widehat U^\ep_{\rho_{\bf n},\bn}(.,t)-\textbf{u}(.,t)\|^2_{H^1(\Omega)} \nn\\
						&\le   \frac{16}{A_1-A_0}  \left( 1+3Q_{\bf n}^2\right)\int_t^T 	e^{2\rho_{\bf n}(\tau-T)}	{\bf E}\|\widehat U^\ep_{\rho_{\bf n},\bn}(.,\tau)-\textbf{u}(.,\tau)\|^2_{L^2(\Omega)}   d\tau\nn\\
						&+{\bf E}\|\widehat H_{\beta_{\bf n}}-H\|^2_{H^1(\Omega)}+ \frac{16T}{A_1-A_0} {\bf E} \lf\|\widehat G_{\beta_{\bf n}}- G\rt\|^2_{L^\infty(0,T;L^2(\Omega))}\nn\\
						&+\frac{16 TA_1^2}{A_1-A_0}  \rho_{\bf n}^{-2\gamma} e^{-2T\rho_{\bf n} }\|\textbf{u}\|_{\kg}^2   +\frac{16}{A_1-A_0} \ep^2   (T-t) \|{\bf u}\|^2_{L^\infty (0,T);H^{2}(\Omega))}.
						\end{align}	
						Multiplying both sides of the last inequality by $e^{2T\rho_{\bf n}}$, we obtain
						\begin{align}  \label{important1}
						&e^{2\rho_{\bf n} t}	{\bf E}\|\widehat U^\ep_{\rho_{\bf n},\bn}(.,t)-\textbf{u}(.,t)\|^2_{H^1(\Omega)} \nn\\
						&\le   \frac{16}{A_1-A_0}  \left( 1+3Q_{\bf n}^2\right)\int_t^T 	e^{2\rho_{\bf n} \tau}	{\bf E}\|\widehat U^\ep_{\rho_{\bf n},\bn}(.,\tau)-\textbf{u}(.,\tau)\|^2_{L^2(\Omega)}   d\tau\nn\\
						&+e^{2T \rho_{\bf n}} {\bf E}\|\widehat H_{\beta_{\bf n}}-H\|^2_{H^1(\Omega)}+ \frac{16e^{2T \rho_{\bf n}} }{A_1-A_0} {\bf E} \lf\|\widehat G_{\beta_{\bf n}}- G\rt\|^2_{L^\infty(0,T;L^2(\Omega))}\nn\\
						&+\frac{16 TA_1^2}{A_1-A_0}  \rho_{\bf n}^{-2\gamma} \|\textbf{u}\|_{\kg}^2   +\frac{16}{A_1-A_0} \ep^2 e^{2T \rho_{\bf n}}   (T-t) \|{\bf u}\|^2_{L^\infty (0,T);H^{2}(\Omega))}.
						\end{align}	
						\noindent 	Now, we need to find an  upper bound of ${\bf E}\|\widehat H_{\beta_{\bf n}}-H\|^2_{H^1(\Omega)}$. By  Theorem 2.1 \cite{Tuan3}, we get
						\begin{align}   \label{erroraa}
						&\Big\| \widehat H_{\beta_{\bf n}}(x) -H(x)  \Big\|^2_{H^1(\Omega)}\nn\\
						&= 4  \sum_{{\bf p} \in   \mathcal{W}_{  \beta_{\bf n}} } |{\bf p}|^2 \Bigg[\frac{\pi^d}{\prod_{k=1}^d n_k } \sum_{i_1=1}^{n_1}...\sum_{i_d=1}^{n_d} \Lambda_{i_1,i_2,...i_d} \Upsilon_{i_1,i_2,...i_d} \psi_p	( x_{i_1},...x_{i_d}) -{\bf \overline \Gamma}_{{\bf n}, {\bf p}}\Bigg]^2  \nn\\
						&\quad \quad \quad  + 4  \sum_{{\bf p} \in   \mathcal{W}_{  \beta_{\bf n}} }  |{\bf p}|^2 \Big| H_{\bf p}\Big|^2\nn\\
						&\le \underbrace{\frac{8\pi^{2d}}{\left(\prod_{k=1}^d n_k\right)^2 } \sum_{{\bf p} \in   \mathcal{W}_{  \beta_{\bf n}}  } |{\bf p}|^2 \Bigg[\sum_{i_1=1}^{n_1}...\sum_{i_d=1}^{n_d} \Lambda_{i_1,i_2,...i_d} \Upsilon_{i_1,i_2,...i_d}  \Bigg]^2}_{:=A_{111}}\nn\\
						&+\underbrace{8 \sum_{{\bf p} \in   \mathcal{W}_{  \beta_{\bf n}}  } |{\bf p}|^2 \Big|{\bf \overline \Gamma}_{{\bf n},{\bf p}} \Big|^2}_{:=A_{222}}+ \underbrace{4  \sum_{{\bf p} \notin   \mathcal{W}_{  \beta_{\bf n}} }  |{\bf p}|^2 \Big| H_{\bf p}\Big|^2}_{:=A_{333}}.
						\end{align}
						The expectation of $A_{111}$ is bounded by
						\begin{align*}
						{\bf  E} A_{111} & \le \frac{8\pi^{2d}}{\left(\prod_{k=1}^d n_k\right)^2 } \sum_{p \in \mathcal{W}_{  \beta_{\bf n}}} \bn \frac{  \prod_{k=1}^d n_k}{\pi^d} V_{\max}^2  = \frac{8 \pi^d}{ \prod_{k=1}^d n_k } V_{\max}^2 \beta_{\bf n}  \text{card} \left( \mathcal{W}_{  \beta_{\bf n}}\right) .
						\end{align*}
						It follows from  Theorem 2.1 \cite{Tuan3}  that
						\begin{align}
						{\bf  E} A_{111}  \le   8\pi^d  V_{max}^2  \frac{2 \pi^{d/2}}{ d \Gamma(d/2)}   \frac{  \beta_{\bf n}^{\frac{d+2}{2}} }{ \prod_{k=1}^d n_k }. \label{errorbb}
						\end{align}
						which we use (2.30) of  \cite{Tuan3}	\begin{align}
						\text{card} \left(  	\mathcal{W}_{  \beta_{\bf n}} \right)  \le \frac{2 \pi^{d/2}}{ d  \Gamma(d/2)} \beta_{\bf n}^{d/2}.  \label{card}
						\end{align}
						From (2.37) of \cite{Tuan3}, we obtain
						\begin{align}
						A_{222}&= 8 \sum_{p \in \mathcal{W}_{  \beta_{\bf n}}} |{\bf p}|^2 \Big|{\bf \overline \Gamma}_{{\bf n},{\bf p}} \Big|^2 \nn\\
						&\le 8  \beta_{\bf n} \mathcal{C}^2 (\mu_1,...\mu_d) \Big\| H \Big\|^2_{\mathcal{H}^{\mu_0}(\Omega) } \prod_{k=1}^d  ( n_k)^{-4\mu_k} \text{card} \left(\mathcal{W}_{  \beta_{\bf n}} \right)\nn\\
						&\le \frac{ 16 \mathcal{C}^2 (\mu_1,...\mu_d) \pi^{d/2}}{ d  \Gamma(d/2)}  \Big\| H \Big\|^2_{\mathcal{H}^{\mu_0}(\Omega) } \beta_{\bf n}^{\frac{d+2}{2}} \prod_{k=1}^d  ( n_k)^{-4\mu_k}.  \label{errorcc}
						\end{align}
						For $A_{333}$ on the right hand side of \eqref{erroraa}, noting that $ {\bf |p|^2}  \ge \beta_{\bf n} $ if $ {\bf p} \notin  \mathcal{W}_{  \beta_{\bf n}}$,  we have the following estimation
						\begin{equation}
						A_{333}=4  \sum_{p \notin  \mathcal{W}_{  \beta_{\bf n}}}  |{\bf p}|^{-2\mu_0} |{\bf p}|^{2\mu_0+2} \Big| H_p\Big|^2 \le 4\beta_{\bf n}^{-\mu_0} \Big\| H \Big\|^2_{\mathcal{H}^{\mu_0+1}(\Omega) }. \label{errordd}
						\end{equation}
						Combining \eqref{erroraa}, \eqref{errorbb}, \eqref{errorcc} and \eqref{errordd},  we deduce that
						\begin{align}
						{\bf E}	\Big\| \widehat H_{\beta_{\bf n}}(x) -H(x)  \Big\|^2_{H^1(\Omega)} &\le 	{\bf E}		A_{111}+	A_{222}+	A_{333} \nn\\
						& \le  8\pi^d  V_{max}^2  \frac{2 \pi^{d/2}}{ d  \Gamma(d/2)}   \frac{  \beta_{\bf n}^{\frac{d+2}{2}} }{ \prod_{k=1}^d n_k }\nn\\
						&+ \frac{ 16 \mathcal{C}^2 (\mu_1,...\mu_d) \pi^{d/2}}{ d \Gamma(d/2)}  \Big\| H \Big\|^2_{\mathcal{H}^{\mu_0}(\Omega) } \beta_{\bf n}^{\frac{d+2}{2}} \prod_{k=1}^d  ( n_k)^{-4\mu_k}+4\beta_{\bf n}^{-\mu_0} \Big\| H \Big\|^2_{\mathcal{H}^{\mu_0+1}(\Omega) },
						\end{align}
						which implies that  $	{\bf E}	\Big\| \widehat H_{\beta_{\bf n}} -H  \Big\|^2_{H^1(\Omega)}$
						is of order
					$
						\max \Bigg( \frac{   \beta_{\bf n}^{\frac{d+2}{2}} }{\prod_{k=1}^d  ( n_k)^{4\mu_k} }, ~\beta_{\bf n}^{-\mu_0}  \Bigg).
						$
						This together with Corollary  \ref{corollary2.1} yields
						\begin{align*}
						&e^{2T \rho_{\bf n}} {\bf E}\|\widehat H_{\beta_{\bf n}}-H\|^2_{H^1(\Omega)}+ \frac{16e^{2T \rho_{\bf n}} }{A_1-A_0} {\bf E} \lf\|\widehat G_{\beta_{\bf n}}- G\rt\|^2_{L^\infty(0,T;L^2(\Omega))}\nn\\
						&+\frac{16 TA_1^2}{A_1-A_0}  \rho_{\bf n}^{-2\gamma} \|\textbf{u}\|_{\kg}^2   +\frac{16}{A_1-A_0} \ep^2 e^{2T \rho_{\bf n}}   (T-t) \|{\bf u}\|^2_{L^\infty (0,T);H^{2}(\Omega))}
						\end{align*}	
						is of order
						\begin{equation}
						\max	\Bigg(  \frac{  e^{2\rho_{\bf n} T}   \beta_{\bf n}^{\frac{d+2}{2}} }{\prod_{k=1}^d  ( n_k)^{4\mu_k} }, e^{2\rho_{\bf n} T} \beta_{\bf n}^{-\mu_0} , \rho_{\bf n}^{-2\gamma}  \Bigg)+ \ep^2 e^{2\rho_{\bf n} T} . \label{important2}
						\end{equation}
						Combining \eqref{important1} and \eqref{important2} gives that 	$
						{\bf E}	\Big\| \widehat U^\ep_{\rho_{\bf n},\bn}-{\bf u} \Big\|_{H^1(\Omega)}
						$
						is of order
						\begin{equation}
							\exp \Big(  \frac{48T Q_{\bf n}^2}{A_1-A_0}  \Big)\max \Bigg( 	  \frac{  e^{2\rho_{\bf n} (T-t)}   \beta_{\bf n}^{\frac{d+2}{2}} }{\prod_{k=1}^d  ( n_k)^{4\mu_k} },   e^{-2\rho_{\bf n} t} \rho_{\bf n}^{-2\gamma}, e^{2\rho_{\bf n} (T-t)}  \bn^{-\mu_0} \Bigg)+	\exp \Big(  \frac{48T Q_{\bf n}^2}{A_1-A_0}  \Big) \ep^2 e^{2\rho_{\bf n} T}.
						\end{equation}
						This implies the desired result  \eqref{er555}.

						\end{proof}

				\end{document}